\documentclass[a4paper,12pt]{article}

\usepackage{amsfonts}
\usepackage{amscd,color}
\usepackage{amsmath,amsfonts,amssymb,amscd}
\usepackage{indentfirst,graphicx,epsfig}
\usepackage{graphicx}
\input{epsf}
\usepackage{graphicx}
\usepackage{epstopdf}
\usepackage{caption}

\newtheorem{thm}{Theorem}[section]
\newtheorem{df}[thm]{Definition}

\newtheorem{conjecture}[thm]{Conjecture}
\newtheorem{lem}[thm]{Lemma}

\newtheorem{cor}[thm]{Corollary}
\newtheorem{pro}[thm]{Proposition}
\newenvironment {proof} {\noindent{\em Proof.}}{\hspace*{\fill}$\Box$\par\vspace{4mm}}
\newcommand{\ml}{l\kern-0.55mm\char39\kern-0.3mm}

\title{\textbf{ Conflict-free connection of
trees\footnote{Supported by NSFC No.11531011.}}}
\author{{\small Hong Chang, Meng Ji, Xueliang Li, Jingshu Zhang } \\
{\small  Center for Combinatorics and LPMC}\\
{\small Nankai University, Tianjin 300071, P.R. China}\\
{\small Email: changh@mail.nankai.edu.cn, jimengecho@163.com,}\\
{\small lxl@nankai.edu.cn, jszhang@mail.nankai.edu.cn}\\
}
\date{}
\begin{document}
\maketitle
\begin{abstract}
We study the conflict-free connection coloring of trees, which is also the conflict-free coloring of the so-called edge-path hypergraphs of trees. We first prove that for a tree $T$
of order $n$, $cfc(T)\geq cfc(P_n)=\lceil \log_{2} n\rceil$, which 
completely confirms the conjecture of Li and Wu. We then
present a sharp upper bound for the conflict-free connection number
of trees by a simple algorithm. Furthermore, we show that the
conflict-free connection number of the binomial tree with $2^{k-1}$
vertices is $k-1$. At last, we study trees which are $cfc$-critical,
and prove that if a tree $T$ is $cfc$-critical, then the conflict-free
connection coloring of $T$ is equivalent to the edge ranking of $T$.\\[2mm]
\end{abstract}

\section{Introduction}

All graphs in this paper are undirected, simple and nontirivial. We
follow \cite{BM} for graph theoretical notation and terminology not
described here. Let $G$ be a graph. We use $V(G), E(G), n(G), m(G)$,
and $\Delta(G)$ to denote the vertex set, edge set, number of
vertices (order of $G$), number of edges (size of $G$), and maximum
degree of $G$, respectively. Let $N(v)$ denote the neighborhood of $v$
in $G$. Given two graphs $G_1$ and $G_2$, the \emph{union} of $G_1$ and
$G_2$, denoted by $G_1\cup G_2$, is the graph with vertex set
$V(G_1)\cup V(G_2)$ and edge set $E(G_1)\cup E(G_2)$. A hypergraph $H$ is
a pair $(\mathcal{V},\mathcal{E})$, where $\mathcal{V}$ is the vertex set
of $H$ and $\mathcal{E}$ is the hyperedge set which is a family of nonempty
subsets of $\mathcal{V}$. The concept of hypergraphs is a generalization of
graphs, because a graph is a hypergraph in which each hyperedge is a pair of
vertices.

A vertex-coloring of a hypergraph $H=(\mathcal{V},\mathcal{E})$
is called \emph{conflict-free} if each hyperedge $e$ of $H$ has a
vertex of unique color that does not get repeated in $e$. The smallest
number of colors required for such a coloring is called the
\emph{conflict-free chromatic number} of $H$, and is denoted by
\emph{$\chi_{cf}(H)$}. This parameter was first introduced by Even,
Lotker, Ron and Smorodinsky \cite{ELR}, with an emphasis
on hypergraphs induced by geometric shapes. The main application of a
conflict-free coloring is that it models a frequency assignment for
cellular networks. A cellular network consists of two kinds of nodes:
base stations and mobile agents. Base stations have fixed positions
and provide the backbone of the network, they are represented by
vertices in $\mathcal{V}$. Mobile agents are the clients of the
network and are served by base stations. This is done as follows:
every base station has a fixed frequency, this is represented by
the coloring $c$. If an agent wants to establish a link with a base
station, it has to tune itself to this base station's frequency.
Since agents are mobile, they can be in the range of many different
base stations. To avoid interference, the system must assign
frequencies to base stations in the following way: for any range,
there must be a base station in the range with a frequency that is
not used by some other base station in the range. One can solve the
problem by assigning $n$ different frequencies to the $n$ base stations.
However, using many frequencies is expensive, and therefore a scheme that
reuses frequencies, where possible, is preferable. Conflict-free coloring
problems have been the subject of many recent papers due to their
practical and theoretical interest. One can find many results on
conflict-free coloring, see \cite{BCS,CFK,EM,HS,PT1,S}.

Some hypergraphs, which are induced from a given simple graph $G=(V,E)$ and its neighborhoods or its paths, have been studied with respect to the conflict-free coloring are listed below:

(i) \ The vertex-neighborhood hypergraph $H_N(G)=(\mathcal{V},\mathcal{E})$,
which has been studied in \cite{C,PT}, is a hypergraph with
$\mathcal{V}(H_N)=V(G)$ and $\mathcal{E}(H_N)=\{N_G(x)| x \in V(G)\}$.

(ii) \ The vertex-path hypergraph $H_{VP}(G)=(\mathcal{V},\mathcal{E})$ is a
hypergraph with $\mathcal{V}(H_{VP})=V(G)$ and $\mathcal{E}(H_{VP})=\{V(P)|$
$P$ is a path of $G$ $\}$. A conflict-free coloring of $H_{VP}(G)$ is called
a \emph{conflict-free coloring of $G$ with respect to paths}; we also define
the corresponding graph chromatic number, $\chi_{cf}^{P}(G)=\chi_{cf}(H_{VP}(G))$.
In \cite{CT}, the authors proved that it is coNP-complete to decide whether a
given vertex-coloring of a graph is conflict-free with respect to paths.
And in \cite{CKP}, the authors studied the conflict-free coloring of
tree graphs with respect to paths, and they showed that
$\chi_{cf}^{P}(P_n)=\lceil \log_2 (n+1)\rceil $ for a path $P_n$ on $n$
vertices, and $\chi_{cf}^{P}(B^*_{2(r+1)+3r}) \leq 4r+2$ for the complete binary
tree with $5r+2$ levels.

(iii) \ The edge-path hypergraph $H_{EP}(G)=(\mathcal{V},\mathcal{E})$ is a
hypergraph with $\mathcal{V}(H_{EP})=E(G)$ and $\mathcal{E}(H_{EP})=\{E(P)|$
$P$ is a path of $G$ $\}$, which is studied in this paper.

Inspired by rainbow connection colorings \cite{LSS, LS, LS1} and proper
connection colorings \cite{LM, LMQ} of graphs and conflict-free colorings
of hypergraphs, Czap et al. \cite{CJV} introduced the concept of
conflict-free connection colorings of graphs. An edge-colored graph
$G$ is called \emph{conflict-free connected} if each pair of distinct
vertices is connected by a path which contains at least one color used
on exactly one of its edges. This path is called a
\emph{conflict-free path}, and this coloring is called a
\emph{conflict-free connection coloring} of $G$. The
\emph{conflict-free connection number} of a connected graph $G$,
denoted by $cfc(G)$, is the smallest number of colors required to
color the edges of $G$ so that $G$ is conflict-free connected.
A conflict-free connection coloring of a connected graph $G$
using $cfc(G)$ colors is called an \emph{optimal conflict-free connection
coloring} of $G$. It is easy to see that for a tree $T$, a conflict-free
connection coloring of $T$ is a proper edge-coloring, and
$cfc(T)\geq \chi'(T)=\Delta(T)$. Note that a conflict-free coloring
of the edge-path hypergraph $H$ of a graph $G$ is a conflict-free
connection coloring of $G$. The other way round is not true in general,
since some pairs of vertices in a general graph $G$ may have more than
one path between them. However, for any tree $T$, the conflict-free
coloring of the edge-path hypergraph of $T$ is equivalent to the
conflict-free connection coloring of $T$, i.e.,
$\chi_{cf}(H_{EP}(T)) = cfc(T)$, since every pair of vertices in $T$ has
a unique path between them. In this paper, we study the conflict-free
connection coloring of trees, which is also the conflict-free coloring of
edge-path hypergraphs of trees. At first, we present some results on the
conflict-free connection for general graphs.

\begin{lem}{\upshape \cite{CHLMZ}}\label{lem1-1}
Let $G$ be a connected graph of order $n$.
Then $1\leq cfc(G)\leq n-1$.
Moreover, $cfc(G)=1$ if and only if $G=K_n$,
and $cfc(G)=n-1$ if and only if $G=K_{1,n-1}$.
\end{lem}

A \emph{$k$-edge ranking} of a connected graph $G$ is a labeling of its edges
with labels $1,\ldots,k$ such that every path between two edges with the same
label $i$ contains an edge with label $j > i$. The minimum number $k$ needed
for a $k$-edge ranking of $G$ is called the \emph{edge ranking number} and
denoted by $rank(G)$. An edge ranking is \emph{optimal} if it uses $rank(G)$
distinct labels. The edge ranking problem has been studied by a number of
researchers as they found applications in different context
\cite{BDJKKMT,IRV, TGS}. This problem is now known to be NP-hard for general
graphs \cite{LY}. Nonetheless, in most applications, the graphs concerned are restricted to trees only, this initiates the study of edge ranking of trees.

With respect to trees, Iyer et al. \cite{IRV} first gave an approximation
algorithm for finding an edge ranking with size at most twice the optimal.
de la Torre et al. \cite{TGS} are the first to devise a polynomial time
algorithm for finding an optimal edge ranking, which takes
$O(n^3\log n)$ time. Then there are several other papers in the literature on developing faster algorithms for the edge ranking problem, and the best known algorithm requires linear time \cite{LY2}.

\begin{lem}{\upshape \cite{LY2}}\label{lem1-0}
An optimal edge ranking of a tree with $n$ vertices can be computed
in $O(n)$ time.
\end{lem}

In \cite{CJV}, the authors gave the relationship between the conflict-free
connection number and the edge ranking number of general graphs.

\begin{lem}{\upshape \cite{CJV}}\label{lem1-4}
If $G$ is a connected graph, then $cfc(G) \leq rank(G)$.
\end{lem}

\begin{thm}{\upshape \cite{CJV}}\label{thm1}
If $G$ is a noncomplete $2$-connected graph, then $cfc(G)=2$.
\end{thm}

In \cite{CHLMZ}, the authors weaken the condition of the above theorem
and got the following result.

\begin{thm}{\upshape \cite{CHLMZ, Deng}}\label{thm2}
Let $G$ be a noncomplete $2$-edge-connected graph. Then $cfc(G)=2$.
\end{thm}

Let $C(G)$ be the subgraph of $G$ induced by the set of cut-edges of
$G$. It is easy to see that every component of $C(G)$ is a tree and hence
$C(G)$ is a forest. Let
$h(G)=\text{max}\{cfc(T) : \text{T is a component of}\ C(G)\}$.
For a graph $G$ with cut-edges, the authors of \cite{CJV} gave lower and upper
bounds of $cfc(G)$ in terms of $h(G)$.

\begin{thm}{\upshape \cite{CJV}}\label{thm3}
If $G$ is a connected graph with cut-edges, then
$h(G)\leq cfc(G)\leq h(G)+1$.
Moreover, the bounds are sharp.
\end{thm}

Recently, the authors in \cite{CHLMZ} gave a sufficient condition
such that the lower bound is sharp for $h(G)\geq 2$.

\begin{thm}{\upshape \cite{CHLMZ}}\label{thm4}
Let $G$ be a connected graph with $h(G)\geq 2$. If there exists a
unique component $T$ of $C(G)$ satisfying $(i)$ $cfc(T)=h(G)$, $(ii)$
$T$ has an optimal conflict-free connection coloring such that
one color is used only once, then $cfc(G)=h(G)$.
\end{thm}

So, the problem of determining the value of $cfc(G)$ for a graph $G$
without bridges or cut-edges is completely solved. The rest graphs all
have cut-edges. Such extremal graphs are trees for which every edge is
a cut-edge. And by the above theorem, to determine the conflict-free
connection number of general graphs relies on determining the conflict-free
connection number of trees, with an error of only one. Next, we present some
known results on the conflict-free connection number of trees.

\begin{lem}{\upshape \cite{CJV}}\label{lem1-2}
If $P_n$ is a path on $n$ vertices, then $cfc(P_n)=\lceil
\log_2 n\rceil $.
\end{lem}

\begin{lem}{\upshape \cite{CJV}}\label{lem1-3}
If $T$ is a tree on $n$ vertices with maximum degree
$\Delta(T)\geq 3$ and diameter $diam(T)$, then
$$
\max\{\Delta(T),\log_2 diam(T)\}\leq cfc(T)\leq \frac{(\Delta(T)-2)\log_2 n}{\log_2\Delta(T)-1}.
$$
\end{lem}

The following result indicates that when the maximum degree of a
tree is large, the conflict-free connection number is
immediately determined by its maximum degree.

\begin{thm}{\upshape \cite{CHLMZ}}\label{thm5}
Let $T$ be a tree of order $n$, and $t$ be a positive number such that
$t\leq \frac {n-2} {2}$. Then $cfc(T)=n-t$ if and only if $\Delta(T)=n-t\geq \frac {n+2} {2}$.
\end{thm}

It is easy to obtain the following result for trees with diameter
$3$.

\begin{lem}\label{lem1-6}
Let $S_{a,b}$ be a tree with diameter
$3$ such that the two non-leaf vertices have degrees $a$ and $b$, Then $cfc(S_{a,b})=\Delta(S_{a,b})$.
\end{lem}
\begin{proof} By Lemma \ref{lem1-3}, we have
$cfc(S_{a,b})\geq\Delta(S_{a,b})=\max\{a,b\}$.
It remains to prove the matching lower bound.
Without loss of generality, we assume that the non-leaf
vertex $u$ has maximum degree $a$ and the other non-leaf
vertex $v$ has degree $b$. We provide an edge-coloring of
$S_{a,b}$ with $a$ colors: assign $a$ distinct colors to the
edges incident with $u$, then for the remaining edges, assign
$b-1$ distinct used colors which do not contain the color
on the edge $uv$. It is obvious that this is a conflict-free
connection coloring of $S_{a,b}$.
Thus, $cfc(S_{a,b})\leq\Delta(S_{a,b})=\max\{a,b\}$.
\end{proof}

Recently, Li and Wu proposed the following conjecture in \cite{LW}.

\begin{conjecture}{\upshape \cite{LW}}\label{conj1-4}
For a tree $T$ of order $n$,
$cfc(T)\geq cfc(P_n)=\lceil \log_{2} n\rceil$.
\end{conjecture}

\begin{df}\label{df1-1}
A tree $T$ is called \emph{$cfc$-critical} if for every proper
subtree $T'$ of $T$, $cfc(T')<cfc(T )$, which means that for
every edge $e$ of $T$, any one of the (two) nontrivial components
in $T-e$ has a conflict-free connection number less than $cfc(T )$.
In addition, we say that a tree $T$ is called \emph{$k$-$cfc$-critical}
if $T$ is $cfc$-critical and $cfc(T )=k$.
\end{df}

Let us give an overview of the results of this paper. In Section $2$,
we present a sharp lower bound for the conflict-free connection
number of trees, this completely confirms Conjecture \ref{conj1-4}.
In Section $3$, we give a sharp upper bound for the conflict-free
connection number of trees by a simple algorithm we develop.
Furthermore, we show that the conflict-free connection number of the
binomial tree with $2^{k-1}$ vertices is $k-1$. In Section $4$, we
study trees which are $cfc$-critical, and prove that if a tree $T$ is
$cfc$-critical, then the conflict-free connection coloring of $T$ is
equivalent to the edge ranking of $T$.

\section{The lower bound}

In order to obtain a lower bound for the conflict-free connection
number of trees, we define a new kind of edge-colorings of graphs,
which is regarded as the generalization of the conflict-free connection
colorings of graphs.

\begin{df}\label{df2-1}
Let $G$ be a nontrivial connected graph with an edge-coloring.
A path in $G$ is called an \emph{odd path} if there is a color
that occurs an odd number of times on the edges of the path.
An edge-colored graph $G$ is called \emph{odd connected}
if any two distinct vertices of $G$ are connected by an odd path,
and this coloring is called an
\emph{odd connection coloring} of $G$.
For a connected graph $G$, the minimum number of colors
that are required in order to make $G$ odd connected
is called the \emph{odd connection number} of $G$,
denoted by \emph{$oc(G)$}.
\end{df}

It is easy to see that every conflict-free connection coloring of $G$
is an odd connection coloring, which implies the following easy result.

\begin{pro}\label{pro2-2}
Let $G$ be a connected graph, then $oc(G)\leq cfc(G)$.
\end{pro}

The parity vector below is a very useful tool to study the
odd connection coloring of graphs.

\begin{df}\label{df2-3}
Given an edge-coloring $c:E\rightarrow \{1,\cdots,k\}$ of $G$
and a path $P$ of $G$, the parity vector of $P$ is an element
of $\{0,1\}^k$ in which the $i$-th coordinate equals the parity
($0$ for even, or $1$ for odd) of the number of edges in $P$
with color $i$.
\end{df}

It is clear that an edge-colored graph $G$ is odd connected
if and only if any two distinct vertices of $G$ are connected by
a path whose parity vector is not the all-zero vector. Now we are
ready to give the proof of our main result of this section. The idea
in the following lemma was also used in \cite{CKP} for the
vertex-coloring case and we show here how to transfer this idea to the
edge-coloring.

\begin{lem}\label{lem2-4}
Let $T$ be a tree of order $n$. Then $oc(T)\geq \lceil \log_{2} n\rceil $.
\end{lem}

\begin{proof}
Let $c$ be an odd connection coloring of $T$ with $oc(T)$ colors.
Take $n-1$ distinct paths each starting from a fixed leaf vertex of $T$ to
the other $n-1$ vertices. First, we claim that any two of these $n-1$ paths
have distinct parity vectors. Assume, to the contrary, two of them,
say $P_1$ and $P_2$, have the same parity vector $\mathbf{a}$. Let $E'$
be the symmetric difference of $E(P_1)$ and $E(P_2)$. Then the subgraph
induced by $E'$ is also a path $P'$. Notice that the parity vector
$\mathbf{a'}$ of $P'$ is $\mathbf{a}+\mathbf{a}=\mathbf{0}$, as the colors
in $E(P_1) \cap E(P_2)$ are counted twice, which cannot change the value
of $\mathbf{a'}$. So, on the path $P'$ each color appears an even number of
times, a contradiction. Thus, there are $n-1$ distinct parity vectors,
none of which is the all-zero vector. On the other hand, the number of
non-zero parity vectors is at most $2^{oc(T)}-1$, which implies that
$2^{oc(T)}-1\geq n-1$, and hence $oc(T)\geq \lceil \log_{2} n\rceil $.
\end{proof}

The following result is an immediate consequence of Lemmas
\ref{lem1-2} and \ref{lem2-4} and Proposition \ref{pro2-2}, which
completely confirms Conjecture \ref{conj1-4}.

\begin{thm}\label{thm2-5}
Let $T$ be a tree of order $n$.
Then $cfc(T)\geq cfc(P_n)= \lceil \log_{2} n\rceil $.
\end{thm}

\section{An algorithm for the upper bound}

At the very beginning of this section, we give the following
concept.

\begin{df}\label{df3-0}
Let $T$ be a tree. An edge $e$ of $T$ is called \emph{balanced}
if the difference of the sizes of two resulting subtrees
of $T-e$ is minimum among the edges of $G$.
\end{df}

We below present an algorithm for constructing a conflict-free connection
coloring of a given tree. This algorithm starts from the single connected
component $T$, and in each iteration, removes one balanced edge from each
generated subtree of $T$ to split it into two generated subtrees, until all
the generated subtrees become singletons as illustrated in Fig. 1.

For convenience, the \emph{depth} of a tree $T$, denoted by $d(T)$, is
defined as the number of the iterations of the following algorithm;
the \emph{depth} of an edge $e$ in $T$, denoted by $d(e)$, is defined
as the sequence number of the iteration that deletes $e$. That is,
the edges deleted in $d$-th iteration are the ones with depth $d$.

\begin{table}[!htb]
\begin{tabular}{lll}
\hline
{\textbf{Algorithm $1$ for conflict-free connection coloring of a tree}}\\
\hline
\textbf{Input:} A tree $T=(V,E)$.\\

\textbf{Output:} A conflict-free connection coloring $c$ of $T$.\\

\textbf{Step 1:} Initialization: set $F=T$, $c:E\rightarrow\{0\}$.\\
\textbf{Step 2:} Determine whether there exists a component which has more than one vertex \\
in $F$. If so, go to Step 3; otherwise, go to Step 4.\\
\textbf{Step 3:} Choose all the components which have more than one vertex, and then delete \\
a balanced edge $e$ from each of such components to get new components which form
a \\
forest $F'$. Replace $F$ by $F'$. Update $c$: color edges $e$ with $d(e)$. Go to Step 2.\\
\textbf{Step 4:} Return $c$. \\
\hline
\end{tabular}
\end{table}

\begin{figure}[!htb]
\centering
\includegraphics[width=0.4\textwidth]{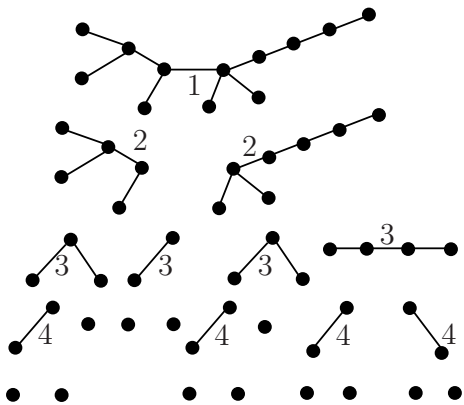}
\caption{A tree of depth $4$ }
\end{figure}

\begin{lem}\label{lem3-0}
Let $T$ be a tree of order $n$. Then
$\max \{\Delta(T),\lceil \log_{2} n\rceil \}\leq d(T)\leq n-1 $.
\end{lem}

\begin{proof} It is obvious that $d(T)\leq n-1 $. And since we
remove only one balanced edge from a generated subtree of $T$ in
Algorithm $1$, it follows that edges in the same depth are not
adjacent, which means $d(T)\geq \Delta(T) $. Thus, we only need
to prove $ d(T)\geq \lceil \log_{2} n\rceil $. Note that the number
of edges with depth $i$ is at most $2^{i-1}$. Algorithm $1$ is
not terminated until all the generated subtrees become singletons.
It follows that $m(T)\leq 2^0+2^1+\cdots+2^{d(T)-1}=2^{d(T)}-1$,
which implies $d(T)\geq \lceil \log_{2} n\rceil $.
\end{proof}

Note that the choice of the balanced edge of each generated subtree
is unique for $K_{1,n-1}$ and $S_{a,n-a}$ by symmetry in Algorithm $1$,
and it is easy to check that $d(K_{1,n-1})=n-1$ and
$d(S_{a,n-a})=\Delta(S_{a,n-a})$. Next, we study the depths of other
trees. Obviously, the path $P_n$ has the same property as $K_{1,n-1}$
and $S_{a,n-a}$.

\begin{lem}\label{lem3-1}
Let $P_n$ be a path on $n$ vertices. Then
$d(P_n)= \lceil \log_{2} n\rceil $.
\end{lem}

\begin{proof} By Lemma \ref{lem3-0}, we have
$d(P_n)\geq \lceil \log_{2} n\rceil $.
Thus, it remains to verify the matching lower bound.
We use induction on $n$.
The statement is evidently true for $n=1$ and $n=2$.
Let $P_n$ be a path on $n$ vertices.
Then we delete a central edge from $P_n$.
The resulting paths $P$ and $P'$ have
at most $\lceil \frac{n}{2}\rceil$ vertices. Therefore,
by the induction hypothesis,
$\max\{d(P), d(P')\}\leq \lceil \log_{2} \frac{n}{2}\rceil $.
Thus, $d(P_n)\leq 1+\max\{d(P),d(P')\}\leq 1+\lceil \log_{2} \frac{n}{2}
\rceil \leq \lceil \log_{2} n\rceil $.
\end{proof}

The (rooted) binomial tree $B_k$ with $2^{k-1}$ vertices is defined
as follows: $B_1$ is a single vertex; for $k>1$, $B_k$ consists of
two disjoint copies of $B_{k-1}$ and an edge between their
two roots, where the root of $B_k$ is the root of the first copy.
These trees are used in \cite{BBJK,GS}. The binomial tree $B_k$ is
another tree class for which the choice of the balanced edge of
each generated subtree is unique in Algorithm $1$.

\begin{lem}\label{lem3-3}
Let $B_k$ be the binomial tree with $2^{k-1}$ vertices for $k\geq 2$.
Then $d(B_k)=k-1$.
\end{lem}

\begin{proof} By Lemma \ref{lem3-0}, we have
$d(B_k)\geq \lceil \log_{2} n\rceil=k-1$.
For the converse, we use induction on $k$.
Noticing that $B_2$ is a path of order $2$,
it follows that the result holds trivially for $k=2$.
Suppose that the result holds for $B_{k-1}$ for $k\geq 3$.
Let $B_k$ be the binomial tree with $2^{k-1}$ vertices, it follows
that $B_k$ consists of two disjoint copies of $B_{k-1}$ and
an edge $e_0$ between their two roots. We delete the edge $e_0$
from $B_k$. Therefore, by the induction hypothesis,
$d(B_{k-1}) \leq k-2$. Thus, $d(B_k)\leq 1+d(B_{k-1})\leq k-1 $.
\end{proof}

The correctness of Algorithm $1$ is confirmed by the following
theorem, and this gives a sharp upper bound for the conflict-free
connection number of trees. Note that for general trees, the
choice of the balanced edge of each generated subtree may not
be unique in Algorithm $1$, which implies that there exist many
priorities of removing edges of $T$.

\begin{thm}\label{thm3-4}
Algorithm $1$ constructs a conflict-free connection coloring of a
given tree $T$. Moreover, $cfc(T)\leq \min \{d(T)|$ $d(T)$ is the
depth of $T$ by a certain priority of removal edges of $T$ in Algorithm
$1$$\}$, where all possible priorities of removing edges of $T$ are taken.
\end{thm}

\begin{proof}
It is sufficient to show that there exists a conflict-free path for each
pair of distinct vertices $u,v$ of $T$ under the coloring given by
Algorithm $1$. Since adjacent edges are colored with distinct colors in
Algorithm $1$, we may assume that $u$ and $v$ are not adjacent.
Let $e_0$ be the balanced edge of the first generated subtree that
simultaneously contains $u$ and $v$ starting from $u$ and $v$, respectively.
Since the color of $e_0$ on the path between $u$ and $v$ is minimal and unique,
it follows that the path between $u$ and $v$ is conflict-free.
Thus, this is a conflict-free connection coloring of $T$,
which implies that the result holds.
\end{proof}

\noindent\textbf{Remark:} Algorithm $1$ provides an optimal
conflict-free connection coloring for the path, the star,
the tree with diameter $3$, and the binomial tree which is
proved below. These imply that the upper bound in Theorem
\ref{thm3-4} is sharp for some classes of graphs.

Let $D(T)=\min \{d(T)|$ $d(T)$ is the depth of $T$ by a certain
priority of removing edges of $T$ in Algorithm $1$$\}$, where all
possible priorities of removing edges of $T$ are taken.
%In the section 4 we have showed that $cfc(R_k)=k$. We found that if $k\leq 6$, $cfc(R_k)\leq D(R_k)\leq cfc(R_k)+1$; If $k\geq 7$, $D(R_k)\geq cfc(R_k)+2$. But when $k$ gradually increasing, $D(R_k)$ will not decrease, however, it will no more than 2$cfc(R_k)$.
Let $A_k$ be the graph obtained from $K_{1,k}$ and $P_{2^{k-1}}$ by
identifying a leaf vertex in $K_{1,k}$ with an end vertex in $P_{2^{k-1}}$ for sufficiently large $k$. It is not hard to see that $cfc(A_k)=k$, but $k<D(A_k)<2k$. For general trees, we propose the following conjecture.

\begin{conjecture}\label{conj3-7}
For a given tree $T$, $D(T)< 2cfc(T)$.
\end{conjecture}

Next, we determine the conflict-free connection number of the
binomial tree. Combining Theorems \ref{thm2-5} and \ref{thm3-4}
and Lemma \ref{lem3-3}, we get the following conclusion.

\begin{thm}\label{thm3-5}
Let $B_k$ be the binomial tree with $2^{k-1}$ vertices,
then $cfc(B_k)=k-1$.
\end{thm}

Recall that the complete binary tree $B_k^*$ has $k$ levels
and $2^k-1$ vertices. It is not difficult to prove by induction
that $B_k^*\subseteq B_{2k-1}$, and so we get the following corollary.

\begin{cor}\label{cor3-6}
Let $B_k^*$ be the complete binary tree with $k$ levels.
Then $k \leq cfc(B_k^*)\leq 2k-2$.
\end{cor}

\begin{proof}
Since $B_k^*$ has $2^k-1$ vertices, it follows that
$cfc(B_k^*)\geq \lceil \log_{2} (2^k-1)\rceil=k$ by
Theorem \ref{thm2-5}. We only need to prove the upper bound.
Since $B_k^*\subseteq B_{2k-1}$, it follows that a conflict-free
connection coloring of $B_{2k-1}$ restricted on the edges of
$B_k^*$ is conflict-free connected. Thus,
$cfc(B_k^*)\leq cfc(B_{2k-1})=2k-2$ by Theorem \ref{thm3-5}.
\end{proof}

\section{$cfc$-critical}

The following theorem indicates that when a tree $T$ is
$cfc$-critical, the conflict-free connection coloring
of $T$ is equivalent to the edge ranking of $T$.

\begin{thm}\label{thm4-0}
Let $T$ be a $cfc$-critical tree, then $cfc(T)=rank(T)$.
\end{thm}
\begin{proof}
By Lemma \ref{lem1-4}, we obtain that $cfc(T)\leq rank(T)$. It remains
to verify the converse. We use induction on the order of $T$. Let $v_0$ be
a leaf vertex of $T$, and $e_0$ be the unique edge incident with $v_0$ in
$T$. Observe that $rank(T)\leq rank(T-v_0)+1$. By the induction hypothesis,
we have that $rank(T-v_0)\leq cfc(T-v_0)$. Since $T$ is $cfc$-critical,
it follows that $cfc(T-v_0)= cfc(T)-1$. Thus,
$rank(T)\leq rank(T-v_0)+1\leq cfc(T-v_0)+1=cfc(T)$.
\end{proof}

The following result is an immediate corollary of Lemma \ref{lem1-0} and
Theorem \ref{thm4-0}.

\begin{cor}\label{cor4-0}
An optimal conflict-free connection coloring of a $cfc$-critical tree with $n$ vertices
can be computed in $O(n)$ time.
\end{cor}

Next, we give some sufficient conditions such that a tree is
$k$-$cfc$-critical. The following two results are trivially
obtained by Lemmas \ref{lem1-1} and \ref{lem1-2}.

\begin{pro}\label{pro4-1}
The star $K_{1,k}$ of order $k+1$ is $k$-$cfc$-critical.
\end{pro}

\begin{pro}\label{pro4-2}
The path $P_{2^{k-1}+1}$ on $2^{k-1}+1$ vertices is $k$-$cfc$-critical.
\end{pro}

Next, we construct two other kinds of trees which are
$k$-$cfc$-critical for every integer $k\geq 2$.

\begin{thm}\label{thm4-3}
Let $Q_k$ be the graph obtained from two copies of $K_{1,k-1}$ with
$k\geq 2$ by identifying a leaf vertex in one copy with a leaf vertex in
the other copy. Then $Q_k$ is $k$-$cfc$-critical.
\end{thm}

\begin{proof} Since $Q_k$ is a path on $2^{k-1}+1$ vertices for $k=2,3$,
it follows that the result holds by Proposition \ref{pro4-2}.
Next, we may assume that $k\geq4$.
Note that $Q_k$ is a tree of order $2k-1$ with diameter $4$.
Let $w$ be the only vertex of degree $2$, which is adjacent to
$u$ and $v$, and let $N(u)=\{w,u_1,\cdots,u_{k-2}\}$ and
$N(v)=\{w,v_1,\cdots,v_{k-2}\}$.

We first show that $cfc(Q_k)=k$.
Suppose that there exists a conflict-free connection coloring of $Q_k$
with at most $k-1$ colors. It follows that all the edges incident with
each of $u$ and $v$ are assigned distinct colors. Without loss of
generality, we assume that the color of the edge $uw$ is $c_i$ and
the color of the edge $vw$ is $c_j$. Thus, there exists a path of
length $4$ having the color sequence $c_i,c_j,c_i,c_j$, a
contradiction. Thus, $cfc(Q_k)\geq k$. Next we define an edge-coloring
of $Q_k$ with $k$ colors: assign $k$ distinct colors to $vw$ and all
the edges incident with $u$, then for the remaining edges assign $k-2$
distinct used colors, none of which is the color on the edge $vw$.
It is obvious that $Q_k$ is conflict-free connected under the coloring.
Thus, $cfc(Q_k)=k$.

Next, we prove that $Q_k$ is $k$-$cfc$-critical. It suffices to show that
for every edge $e$ of $Q_k$, each nontrivial component in $Q_k-e$ has a
conflict-free connection number less than $k$. Suppose that $e$ is one of
the edges $uw$ and $vw$. Then the resulting subtrees of $Q_k-e$ are
$K_{1,k-2}$ and $K_{1,k-1}$. It follows that the result holds by
Lemma \ref{lem1-1}. Suppose that $e$ is a pendent edge of $Q_k$. We may
assume that $e=uu_i$ (for some $i$ with $1\leq i\leq k-2$) is incident with
$u$ by symmetry. Notice that the resulting subtrees in $Q_k-e$ are
a singleton $\{u_i\}$ and $Q_k-u_i$. We provide an edge-coloring of
$Q_k-u_i$ as follows: color edges incident with $v$ using $k-1$ distinct
colors, and color edges incident with $u$ using $k-2$ distinct used colors,
which do not contain the color on the edge $vw$. It can be checked
that this is a conflict-free connection coloring of $Q_k-u_i$ with $k-1$
colors. Thus, $cfc(Q_k-u_i)\leq k-1$, and so each nontrivial component in
$Q_k-e$ has a conflict-free connection number less than $k$ for every edge
$e$ of $Q_k$.
\end{proof}

\begin{lem}\label{lem4-4}
Let $R_k$ be the graph obtained from $K_{1,k-1}$ and $P_{2^{k-2}+1}$ with
$k\geq2$ by identifying a leaf vertex in $K_{1,k-1}$ with an end vertex in
$P_{2^{k-2}+1}$. Then $cfc(R_k)=k$.
\end{lem}

\begin{proof}
Firstly, we prove that $cfc(R_k)\geq k$. Note that $R_2$ is a path on $3$
vertices, and $cfc(R_2)=2> 1$. Then we focus on $k\geq 3$. Assume to the
contrary that there is a $R_k$ such that $cfc(R_k)\leq k-1$, and let $k_0$
be the minimum $k$ with such property. Suppose that $T_1$ is the copy
of $K_{1,k_0-1}$ in $R_{k_0}$, $T_2$ is the copy of $P_{2^{k_0-2}+1}$ in
$R_{k_0}$ that has one common leaf vertex with $T_1$.
Let $V(T_1)=\{u,w, u_1,\cdots,u_{k_0-2}\}$ and $V(T_2)=\{w=v_0,v_1\cdots,v_{2^{k_0-2}}\}$, where $u$ is the only vertex
of maximum degree $k_0-1$ in $T_1$, $v_{i}v_{i+1}$ is an edge of $T_2$ for $i=0,\cdots,2^{k_0-2}-1$, and $V(T_1)\cap V(T_2)=w$.
Let $c$ be a conflict-free connection coloring of $R_{k_0}$ with $k_0-1$
colors. Then we present the following claim.

\textbf{Claim 1.} There exist exactly two colors $c_1$, $c_2$ each of which
is used on exactly one of the edges of $T_2$.

\textbf{Proof of Claim 1:} To make $T_2$ conflict-free connected,
it needs $k_0-1$ colors by Lemma \ref{lem1-2}, and there exists one color
$c_1$ used on exactly one of the edges of $T_2$. Suppose that there exists
a unique color $c_1$ used on exactly one of the edges of $T_2$. Note that
the edges of $T_1$ are colored with $k_0-1$ colors. We assume that $uu_i$
(or $uw$) is colored with $c_1$ for some $i$ with $1\leq i\leq k_0-2$. It
follows that there is no conflict-free path from $v_{2^{k_0-2}}$ to $u_i$
(or $u$), which is a contradiction. Thus, there exist at least two colors
$c_1$, $c_2$ each of which is used on exactly one of the edges of $T_2$.
Suppose that there exists another color $c_3$ used on exactly one of the
edges of $T_2$. Without loss of generality, assume that the edge colored
with $c_1$ appears between the edges colored with $c_2$, $c_3$. Then we can
get a conflict-free connection coloring of $T_2$ with $k_0-2$ colors obtained
from the above coloring $c$ by replacing $c_3$ with $c_2$, which is impossible.
Thus, there exist exactly two colors $c_1$, $c_2$ each of which is used on
exactly one of the edges of $T_2$.

\textbf{Claim 2.} The color on the edge $uw$ is neither
$c_1$ nor $c_2$.

\textbf{Proof of Claim 2:} By contradiction. Without loss of
generality, assume that the color on the edge $uw$ is $c_1$, and the color
on the edge $uu_\ell$ is $c_2$ for some $\ell$ with $1\leq \ell \leq k_0-2$.
Then there is no conflict-free path from $v_{2^{k_0-2}}$ to $u_\ell$, which
is a contradiction.

Let $T_2'$ be the subpath resulting from the removal of the two edges with
the colors $c_1$, $c_2$ in $T_2$ such that $v_{2^{k_0-2}}\in T_2'$. Since
the edges of $T_2'$ are colored with at most $k_0-3$ colors, it follows that
$n(T_2')\leq 2^{k_0-3}$ by Theorem \ref{thm2-5}. Let $T_2''$ be the subpath
starting from $w$ with order $2^{k_0-3}+1$ in $T_2-V(T_2')$, and $T_1'$ be
the subtree obtained from $T_1$ by deleting the edge with color $c_2$, where
the edge of $T_2$ with color $c_1$ is closer to $w$ than the edge with color
$c_2$. It is obtained that $T'=T_1'\cup T_2''$ is exactly $R_{k_0-1}$, and the coloring $c$ restricted on the edges of $T'$ is a conflict-free connection
coloring with $k_0-2$ colors. It follows that $cfc(R_{k_0-1})\leq k_0-2$,
which contradicts the minimality of $k_0$. Thus, $cfc(R_k)\geq k$.
And we provide an edge-coloring of $R_k$ with $k$ colors: color the edges
of $T_2-w$ with $k-2$ distinct colors $1,\cdots,k-2 $ such that it is
conflict-free connected, and color the remaining edges of $R_k$ with
$k$ distinct colors $1,\cdots,k$ such that the color on
the edge $wv_1$ is $k$. It is easy to see that $R_k$ is
conflict-free connected under the coloring. Thus, $cfc(R_k)=k$.
\end{proof}

\begin{thm}\label{thm4-5}
Let $R_k$ be the graph obtained from $K_{1,k-1}$ and $P_{2^{k-2}+1}$ with
$k\geq2$ by identifying a leaf vertex in $K_{1,k-1}$ with an end vertex in $P_{2^{k-2}+1}$. Then $R_k$ is $k$-$cfc$-critical.
\end{thm}
\begin{proof}
By Lemma \ref{lem4-4}, we have $cfc(R_k)=k$. We only need to prove that
for every edge $e$ of $R_k$, each nontrivial component in $R_k-e$ has a
conflict-free connection number less than $k$. If $k=2$, then $R_2$ is a
path on $3$ vertices, and the result holds by Proposition \ref{pro4-2}.
Thus, we assume that $k\geq 3$. Suppose that $T_1$ is the copy of
$K_{1,k-1}$ in $R_k$, $T_2$ is the copy of $P_{2^{k-2}+1}$ in $R_k$
that has one common leaf vertex with $T_1$.
Let $V(T_1)=\{u,w, u_1,\cdots,u_{k-2}\}$ and
$V(T_2)=\{w=v_0,v_1\cdots,v_{2^{k-2}}\}$, where $u$ is the only vertex
of maximum degree $k-1$ in $T_1$, $v_{i}v_{i+1}$ is an edge of $T_2$ for $i=0,\cdots,2^{k-2}-1$, and $V(T_1)\cap V(T_2)=w$. In order to obtain
our result, We distinguish the following three cases.

\textbf{Case 1.} $e=uw$.

Note that the resulting subtrees in $R_k-e$ are $K_{1,k-2}$
and $P_{2^{k-2}+1}$. Thus, $R_k$ is $k$-$cfc$-critical by Lemmas
\ref{lem1-1} and \ref{lem1-2}.

\textbf{Case 2.} $e=uu_i$ ($1\leq i \leq k-2$).

It is obtained that the resulting subtrees in $R_k-e$ are
a singleton $\{u_i\}$ and $R_k-u_i$. We provide an edge-coloring
of $R_k-u_i$: color the edges of $T_2-w$ with $k-2$ distinct colors
$1,\cdots,k-2$ such that it is conflict-free connected,
and color the remaining edges with $k-1$ distinct colors
$1,\cdots,k-1$ such that the color on the edge $wv_1$ is $k-1$.
It can be checked that $R_k-u_i$ is conflict-free connected under the
coloring. Thus, $cfc(R_k-u_i)\leq k-1$, which implies that $R_k$
is $k$-$cfc$-critical.

\textbf{Case 3.} $e=v_iv_{i+1}$ ($0 \leq i\leq 2^{k-2}-1$).

Let $T'$ and $T''$ be the resulting subtrees in $R_k-e$, where $u\in T'$.
Obviously, $cfc(T'')\leq k-2$. For $T'$, we define an edge-coloring as
follows: color the edges on the subpath from $w$ to $v_i$ with $k-2$
colors $1,\cdots,k-2$ such that the subpath is conflict-free connected,
and color the remaining edges of $T'$ with $k-1$ distinct colors $1,\cdots,k-1$
such that the color on the edge $uw$ is $k-1$. It is obtained that this is a
conflict-free connection coloring of $T'$. Thus, $cfc(T')\leq k-1$,
and so $R_k$ is $k$-$cfc$-critical.
\end{proof}

\end{document}